\documentclass{amsart}
\usepackage{amssymb}
\usepackage{amsfonts}
\usepackage{graphics}

\font\sr = cmr8
\newtheorem{theorem}{Theorem}
\theoremstyle{plain}

\newtheorem{example}{Example}

\newtheorem{lemma}{Lemma}

\numberwithin{equation}{section}

\begin{document}
\title{Cones and convex bodies with modular face lattices.}
\author{Daniel Labardini-Fragoso}
\author{Max Neumann-Coto}
\author{Martha Takane}
\address{Graduate Program in Mathematics, Northeastern University, Boston,
Ma.}
\address{Instituto de Matem\'{a}ticas, Universidad Nacional Aut\'{o}noma de M%
\'{e}xico, Cuernavaca, M\'{e}xico.}
\keywords{convex, face lattice, modular, hermitian matrix, projective space\\
Research partially supported by PAPIIT-UNAM grant IN103508 and a fellowship
from PASPA}
\dedicatory{Dedicated to Claus M. Ringel on the occasion of his 60th
birthday.}
\maketitle

\begin{abstract}
If a convex body $C$ has modular and irreducible face lattice (and is not
strictly convex), there is a face-preserving homeomorphism from $C$ to a
section of a cone of hermitian matrices over $\mathbb{R}$, $\mathbb{C}$, or $%
\mathbb{H}$, or $C$ has dimension 8, 14 or 26.
\end{abstract}

\section{Introduction.}

Let $C$ be a convex body in $\mathbb{R}^{n}$. A subset $F$ of $C$ is a 
\textit{face} of $C$ if every open interval in $C$ that contains a point of $%
F$ is contained in $F$. An \textit{extreme point} is a 1-point face. If $S$
is any subset of $C$, the\textit{\ face generated by} $S$ is the minimal
face of $C$ containing $S$. The set $\mathcal{F}(C)$ of all faces of $C$
ordered by inclusion is a lattice, where $F\wedge G$ is the intersection of $%
F$ and $G$, and $F\vee G$ is the face generated by $F\cup G$. The lattice $%
\mathcal{L}(C)$ is always \textit{algebraic }(the chains of faces are
finite), \textit{atomic} (faces are generated by extreme points) and \textit{%
complemented} (for every face $F$ there exists a face $G$ such that $F\wedge
G=\emptyset $ and $F\vee G=C$). We want to consider convex bodies for which $%
\mathcal{F}(C)$ is \textit{modular}, i.e. $F\vee (G\wedge H)=(F\vee G)\wedge
H$ whenever $F\leq H$. Modularity is a \ `weak distributivity' property
satisfied by the lattice of normal subgroups of a group and by the lattice
of subspaces of a vector space. For algebraic, atomic lattices, modularity
is equivalent to the existence of a rank function such that $%
rk(F)+rk(G)=rk(F\vee G)+rk(F\wedge G)$ for all $F$ and $G$ \cite{LT}.
Strictly convex bodies and simplices clearly have modular face lattices. No
other polytopes have this property \cite{B}, but there are beautiful
examples of non-polytopal convex bodies in which every pair of extreme
points is contained in a proper face and every pair of faces with more than
one point meet.

If $\mathcal{L}_{1}$ and $\mathcal{L}_{2}$ are lattices, their \textit{%
direct product} is given by $(\mathcal{L}_{1}\times \mathcal{L}_{2},\leq ),$
where $(a,b)\leq (c,d)$ if and only if $a\leq c$ and $b\leq d$. It follows
that the direct product of two lattices is modular if and only if the
factors are modular. A lattice is called \textit{irreducible} if it is not
isomorphic to a direct product of two nontrivial lattices.

If $C_{1}\subset \mathbb{R}^{m}$ and $C_{2}\subset \mathbb{R}^{n}$ are
convex bodies, define $C_{1}\ast C_{2}\subset \mathbb{R}^{m+n+1}$ as the
convex hull of a copy of $C_{1}$ and a copy of $C_{2}$ placed in general
position in the sense that their linear spans are disjoint and have no
common directions. So $C_{1}\ast C_{2}$ is well defined up to a linear
transformation: it is the convex join of $C_{1}$ and $C_{2}$ of largest
dimension. For example $C\ast \left\{ pt\right\} $ is a pyramid with base $C$%
. Let's say that a convex body $C$ is $\ast $-\textit{decomposable} if $%
C=C_{1}\ast C_{2}$ for two convex bodies $C_{1}$ and $C_{2}$.

The natural correspondence (up to linear transformation) between convex
bodies in $\mathbb{R}^{n}$ and closed cones in $\mathbb{R}^{n+1}$ gives an
isomorphism of face lattices in which $C_{1}\ast C_{2}$ corresponds to the
direct product of the cones, so the results of this paper apply to cones.
This project started with the undergraduate thesis of D. Labardini-Fragoso 
\cite{Lab}, who showed that in dimension less than 6 any cone with modular
face lattice is strictly convex or is decomposable (this was conjectured by
Barker in \cite{B}).

\bigskip

\begin{lemma}
A convex body $C$ is $\ast $-decomposable if and only if its lattice of
faces $\mathcal{L}(C)$ is reducible.
\end{lemma}

\begin{proof}
Let $C=C_{1}\ast C_{2}$. Observe that each point $p$ of $C_{1}\ast C_{2}$
with $p\notin C_{i}$, lies in a unique segment joining a point $p_{1}$ of $%
C_{1}$ and a point $p_{2}$ of $C_{2}$. For, if a point lies in two segments $%
p_{1}p_{2}$ and $p_{1}^{\prime }p_{2}^{\prime }$ then the lines $%
p_{1}p_{1}^{\prime }$ and $p_{2}p_{2}^{\prime }$ are parallel or they
intersect, contradicting the assumptions on the spans of $C_{1}$ and $C_{2}$%
. Moreover, if a point $p\ $moves along a straight line in $C_{1}\ast C_{2}$
then the corresponding points $p_{1}$ and $p_{2}$ move along straight lines
in $C_{1}$ and $C_{2}$: If $p$ and $q$ are points in $C$ and $x\in $ $pq$
then $x=tp+(1-t)q=t\lambda p_{1}+t(1-\lambda )p_{2}+(1-t)\mu
q_{1}+(1-t)(1-\mu )q_{2}$ which can be rewritten as a linear combination of
a point in $p_{1}q_{1}$ and a point in $p_{2}q_{2}$ with coefficients adding
up to 1 so $x_{1}\in p_{1}q_{1}$ and $x_{2}\in p_{2}q_{2}$. Now if $%
C_{i}^{\prime }$ is a face of $C_{i}$ then $C_{1}^{\prime }\ast
C_{2}^{\prime }$ is a face of $C_{1}\ast C_{2}$. For, if $x\in C_{1}^{\prime
}\ast C_{2}^{\prime }$ and $x=\lambda p+(1-\lambda )q$ with $p,q\in
C_{1}\ast C_{2}$, then $x_{1}$ lies in $p_{1}q_{1}$ and $x_{2}$ lies in $%
p_{2}q_{2}$ so as $C_{i}^{\prime }$ is a face of $C_{i}$, $p_{i}$ and $q_{i}$
lie in $C_{i}^{\prime }$ so $p$ and $q$ lie in $C_{1}^{\prime }\ast
C_{2}^{\prime }$. Conversely, if $C^{\prime }$ is a face of $C_{1}\ast C_{2}$
and $p\in C^{\prime }$ then $p_{1}$ and $p_{2}$ lie in $C^{\prime }$ so $%
C^{\prime }=(C^{\prime }\cap C_{1})\ast (C^{\prime }\cap C_{2})$. It remains
to show that $C^{\prime }\cap C_{i}$ is a face of $C_{i}$. If $x\in
C^{\prime }\cap C_{1}$ and $x=\lambda p+(1-\lambda )q$ with $p,q\in
C_{1}\ast C_{2}$ then as $C^{\prime }$, and $C_{1}=C_{1}\ast \emptyset $ are
faces of $C_{1}\ast C_{2}$, $p$ and $q$ lie in $C^{\prime }$ and also in $%
C_{1}$, so $C^{\prime }\cap C_{1}$ is a face of $C_{1}$. $C^{\prime }\cap
C_{2}$ is a face of $C_{2}$. So $\mathcal{L}(C_{1}\ast C_{2})\simeq \mathcal{%
L}(C_{1})\times \mathcal{L}(C_{2}).$

If $\mathcal{L}(C)\approx \mathcal{L}_{1}\ast \mathcal{L}_{2}$ then $%
\mathcal{L}_{1}$ and $\mathcal{L}_{2}$ are isomorphic to sublattices of $%
\mathcal{L}(C)$, so $\mathcal{L}_{i}\approx \mathcal{L}(C_{i})$ for two
faces of $C$ with $C_{1}\wedge C_{2}=\varnothing $ and $C_{1}\vee C_{2}=C$.
To show that $C=C_{1}\ast C_{2}$ we need to prove that $span(C_{1})$ and $%
span(C_{2})$ are disjoint and have no directions in common. Suppose that $%
x\in span(C_{1})\cap span(C_{2})$. Take $x_{i}\in Int(C_{i})$ then the line
through $x$ and $x_{i}$ meets $\partial C_{i}$ at two points $a_{i}$ and $%
b_{i}.$ As $a_{2}$ lies in a proper subface $C_{2}^{\prime }$ of $C_{2}$,
the face generated by $C_{1}$ and $a_{2}$ lies in $C_{1}\vee C_{2}^{\prime }$
which is a proper subface of $C_{1}\vee C_{2}$. But the points $a_{1}$, $%
b_{1}$, $a_{2}$, $b_{2}$ determine a plane quadrilateral whose side $%
a_{i}b_{i}$ lies in the interior of $C_{i}$ so its diagonals intersect at an
interior point $c$ of $C_{1}\vee C_{2}$ so the face generated by $C_{1}$ and 
$a_{2}$ (which contains $c$) must be $C_{1}\vee C_{2}$, a contradiction. Now
suppose that $span(C_{1})$ and $span(C_{2})$ have a common direction $v$.
Take $x_{i}\in Int(C_{i})$ then the line through $x_{i}$ in the direction $v$
meets $\partial C_{i}$ at two points $a_{i}$ and $b$. As before $a_{1}$, $%
b_{1}$, $a_{2}$, $b_{2}$ determine a plane quadrilateral whose diagonals
intersect at an interior point $c$ of $C_{1}\vee C_{2}$, but $c$ lies in the
face generated by $C_{1}$ and $a_{2}$ which is a proper face of $C_{1}\vee
C_{2}$.
\end{proof}

\bigskip

Recall that a \textit{projective space} consists of a set $P$ (the points)
and a set $L$ (the lines) so that (a) Each pair of points is contained in a
unique line, (b) If $a,b,c,d$ are distinct points and the lines $ab$ and $cd$
intersect, then the lines $ac$ and $bd$ intersect (c) Each line contains at
least 3 points and there are at least 2 lines (d) Every chain of subspaces
(also called \textit{flats}) has finite length. The maximum length of a
chain starting with a point is the \textit{projective dimension} of the
space.

The flats of a projective space form an algebraic, atomic, irreducible,
modular lattice. Conversely, any lattice with these properties is the
lattice of flats of a projective space, whose points are atoms and whose
lines are joins of two atoms \cite{ATL}. It is a classic result of Hilbert 
\cite{FPG} that a projective space in which Desargues theorem holds is
isomorphic to the projective space $\mathbb{AP}^{n}$ determined by the
linear subspaces of $\mathbb{A}^{n+1}$, for some division ring $\mathbb{A}$,
and that $\mathbb{AP}^{n}$ and $\mathbb{BP}^{m}$ are isomorphic if and only
if $\mathbb{A}$ and $\mathbb{B}$ are isomorphic and $m=n$. All projective
spaces of dimension larger than $2$ are desarguesian, but there are many
non-desarguesian projective planes.

Examples of convex bodies whose face lattices determine the projective
spaces $\mathbb{RP}^{n}$, $\mathbb{CP}^{n}$ and $\mathbb{HP}^{n}$, and the
octonionic projective plane arise as sections of some classical cones.

\begin{example}
Let $\mathbb{F\in \{R}$,$\mathbb{C}$,$\mathbb{H}\}$, let $H_{n}(\mathbb{F})$
be the set of Hermitian (self-adjoint) $n\times n$ matrices with
coefficients in $\mathbb{F}$, and let $C_{n}(\mathbb{F})$ be the subset of
positive-semidefinite matrices ($A$ is positive-semidefinite if $\overline{v}%
Av^{T}\geq 0,$ for all $v\in $ $\mathbb{F}^{n}$). Then $C_{n}(\mathbb{F})$\
is a real cone whose face lattice is isomorphic to the lattice of subspaces
of $\mathbb{F}^{n}$.
\end{example}

To see this, let $A,B\in C_{n}(\mathbb{K})$, and let $\varphi (B)$ denote
the face generated by $B$. Then $A\in \varphi (B)$ if and only if $\ker
A\supseteq \ker B$. For, $A\in \varphi (B)$ $\Leftrightarrow \exists \lambda
>0$ such that $B-\lambda A\in $ $C_{n}(\mathbb{K})\Leftrightarrow \exists
\lambda >0$ such that $\overline{w}Bw^{T}\geq \lambda \overline{w}Aw^{T}\geq
0$ for all $w\in $ $\mathbb{F}^{n}$ $\Longleftrightarrow \overline{w}%
Bw^{T}=0 $ implies $\overline{w}Aw^{T}=0$ for all $w\in $ $\mathbb{F}%
^{n}\Longleftrightarrow \ker A\supseteq \ker B$ (since for $A\in C_{n}(%
\mathbb{K})$, $\overline{w}Aw^{T}=0$ if and only if $Aw^{T}=0$). Therefore $%
\varphi (A)\rightarrow \left( \ker A\right) ^{\perp }$ defines a bijection $\nu$ 
from the set of faces of $C_{n}(\mathbb{K})$ to the set of linear
subspaces of $\mathbb{K}^{n}$.\ To prove that $\nu$ is an isomorphism of
lattices observe that $\nu (\varphi (A)\vee \varphi (B))=\nu (\varphi
(A+B))=\left( \ker \left( A+B\right) \right) ^{\perp }=\left( \ker A\cap
\ker B\right) ^{\perp }=\left( \ker A\right) ^{\perp }\cup \left( \ker
B\right) ^{\perp }$, and on the other hand, if $\varphi (A)\wedge \varphi
(B) $ is a non-empty face, then it is generated by a matrix $C$ with $\ker
C=span(\ker A\cup \ker B)$, so $\nu (\varphi (A)\wedge \varphi (B))=\left( \ker
C\right) ^{\perp }=\left( \ker A\right) ^{\perp }\cap \left( \ker B\right)
^{\perp }$.

\begin{example}
Let $H_{3}(\mathbb{O)}$ be the set of Hermitian $3\times 3$ matrices over $%
\mathbb{O}$. Then the subset $C_{3}(\mathbb{O)}$ of all sums of squares of
elements in $H_{3}(\mathbb{O)}$ is a real cone whose face lattice determines
an octonionic projective plane.
\end{example}

This can be shown using the nontrivial fact that each matrix in $H_{3}(%
\mathbb{O})$ is diagonalizable by an automorphism of $H_{3}(\mathbb{O})$
that leaves the trace invariant \cite{O}, so:\newline
(a) A matrix $A$ in $H_{3}(\mathbb{O)}$ lies in $C_{3}(\mathbb{O)}$ if and
only if it can be diagonalized to a matrix $A^{\prime }$ with non-negative
entries, because if $A$ lies in $C_{3}(\mathbb{O)}$ then $A^{\prime }$ is a
sum of squares of matrices in $H_{3}(\mathbb{O)}$, which have non-negative
diagonal entries. \newline
(b) All the idempotent matrices in $H_{3}(\mathbb{O)}$ lie in $C_{3}(\mathbb{%
O)}$ as they are squares $(A=A^{2})$. The idempotent matrices of trace 1
correspond to the extreme rays of $C_{3}(\mathbb{O)}$ since they can't be
written as non-negative combinations of other idempotent matrices.\newline
(c) Each face of $C_{3}(\mathbb{O)}$ is generated by an idempotent matrix,
because in any cone all the positive linear combinations of the same set of
vectors generate the same face, so a diagonal matrix with non-negative
entries generates the same face as a matrix with only zeros and ones.\newline
(d) Any two idempotent matrices of trace 1 lie in a face generated by an
idempotent matrix of trace 2, because they can be put simultaneously in the
form $\left[ 
\begin{array}{ccc}
a & x & 0 \\ 
\overline{x} & b & 0 \\ 
0 & 0 & 0%
\end{array}%
\right] ,$ and these lie in the face generated by $\left[ 
\begin{array}{ccc}
1 & 0 & 0 \\ 
0 & 1 & 0 \\ 
0 & 0 & 0%
\end{array}%
\right] $.\newline
(e) $A\in C_{3}(\mathbb{O)}$ is an idempotent of trace 1 if and only if $I-A$
is an idempotent with trace 2. If $A$ and $B$ are idempotents of trace 1,
then $A$ lies in the face generated by $I-B$ if and only if $B$ lies in the
face generated by $I-A$. This duality and (d) show that any two faces
generated by idempotent matrices of trace 2 meet in a face generated by an
idempotent matrix of trace 1.

\section{Face lattices defining projective spaces.}

If $C$ is a convex body whose face lattice is modular and irreducible and $C$
is not strictly convex, the set of extreme points of $C$ is a projective
space with flats determined by the faces of $C$. We would like to know which
projective spaces arise in this way, and what convex bodies give rise to
them.

By Blaschke selection theorem \cite{HCG}, the space of all compact, convex
subsets of a convex body in $\mathbb{R}^{n}$, with the Hausdorff metric, is
compact. So the subspace formed by the compact convex subsets of the
boundary is closed, but the subspace formed by the faces is not closed in
general.

\begin{lemma}
\label{compact} If $C$ is a convex body whose face lattice $\mathcal{F}$ is
modular, then the set $\mathcal{F}_{h}$ of faces of rank $h$, with the
Hausdorff metric, is compact for each $h$.
\end{lemma}

\begin{proof}
Let $F_{i}$ be a sequence of faces of rank $h$. Then $F_{i}$ has a
subsequence $F_{i_{j}}$ that is convergent in $\mathcal{C}$, and its limit
is a compact convex set $K$ contained in $\partial C$, so $K$ generates a
proper face $F$ of $C$ of some rank $h^{\prime }$. We claim that $h^{\prime
}=h$ and $K=F$.

If the rank of $F$ was less than $h$, there would be a face $F^{c}$ of $C$
of rank $n-h$ with $F^{c}\cap F=\phi $. As $F$ and $F^{c}$ are two disjoint
compact sets in $\mathbb{R}^{n}$, there exists $\epsilon >0$ such that the $%
\varepsilon $-neighborhoods of $F$ and $F^{c}$ in $\mathbb{R}^{n}$ are
disjoint. But as $F_{i_{j}}\rightarrow K\subset F$ in the Hausdorff metric,
then for sufficiently large $j$, $F_{i_{j}}$ is contained in the $%
\varepsilon $-neighborhood of $F$, therefore $F_{i_{j}}\cap F^{c}=\phi $,
but these 2 faces have ranks that add up to $n$, so they should meet, a
contradiction.

If the rank of $F$ is $h$ and $K\neq F$, there is an extreme point $p\in F-K$%
, and there is a face $F^{\prime }$ of rank $n-h$ that meets $F$ only at $p$%
, so $F^{\prime }\cap K=\phi $ and the previous argument gives a
contradiction.

To show that the rank of $F$ cannot be larger than $h$, proceed inductively
on $n-h$. As a limit of proper faces is contained in a proper face, the
claim holds if $n-h=1$. Given a sequence $F_{i}$ of faces of rank $h$, let $%
F $ be a face generated by the limit of a convergent subsequence $F_{i_{j}}$%
. If $p$ is an extreme point of $C$ not in $F$, then for sufficiently large $%
j$, $p\notin F_{i_{j}}$ (otherwise $p$ would be in $F$). Let $G_{i_{j}}$ be
a face of rank $h+1$ containing $F_{i_{j}}$ and $p$. Now we can assume
inductively that the limit of a convergent subsequence of $G_{i_{j}}$
generates a face of rank $h+1$. This face contains $F$ properly (because it
contains $p$) so $h^{\prime }<h+1$.
\end{proof}

Now recall that a \textit{topological projective space} is a projective
space in which the sets of flats of each rank are given nontrivial
topologies that make the join and meet operations $\vee $ and $\wedge $
continuous, when restricted to pairs of flats of fixed ranks whose join or
meet have a fixed rank \cite{HIG}.

\begin{lemma}
If $C$ is a convex body whose face lattice is modular and irreducible then $%
C $ is strictly convex or\ the set of extreme points $\mathcal{E(}C\mathcal{)%
}$ is a topological projective space which is compact and connected.
\end{lemma}

\begin{proof}
A natural topology for the set of flats is given by the Hausdorff distance
between the faces. By lemma \ref{compact}, $\mathcal{E=F}_{0}$ is compact.
As the lattice is irreducible and has more than 2 points, the 1-flats have
more than 2 points, so (as they are topological spheres) they are connected.
Now every pair of points in $\mathcal{E}$ is contained in one of these
spheres, so $\mathcal{E}$ is connected (one can actually show that each.$%
\mathcal{F}_{h}$ is connected).

It remains to show that $\vee $ and $\wedge $ are continuous on the
preimages of each $\mathcal{F}_{h}$. Suppose $A_{i}\rightarrow A$ , $%
B_{i}\rightarrow B$ where $A_{i}\wedge B_{i}$ and $A\wedge B$ are faces
corresponding to $h$ flats. We need to show that $A_{i}\wedge
B_{i}\rightarrow A\wedge B$. By lemma \ref{compact} $C_{i}=A_{i}\wedge B_{i}$
has convergent subsequences and the limit of a convergent subsequence $%
C_{i_{\alpha }}$ is a face $C_{\alpha }$ corresponding to an $h$ flat. As $%
C_{i_{\alpha }}$ is contained in $A_{i_{\alpha }}$ and $B_{i_{\alpha }}$, $%
C_{\alpha }$ is contained in $A\wedge B$. But $C_{\alpha }$ and $A\wedge B$
are both faces corresponding to $h$ flats, so $C_{\alpha }=A\wedge B$.
Similarly, if $A_{i}\rightarrow A$ , $B_{i}\rightarrow B$ and $A_{i}\vee
B_{i}$ , $A\vee B$ are faces corresponding to $h$ flats, the limit of each
convergent subsequence of $D_{i}=A_{i}\vee B_{i}$ is a face $D$
corresponding to an $h$ flat. As $D_{i}$ contains $A_{i}$ and $B_{i}$, $D$
contains $A\vee B$, and as both faces correspond to $h$ flats they must be
equal.
\end{proof}

Let $C$ be any convex body. Denote by $\mathcal{B}(C\mathcal{)}\subset C$
the set of\ baricenters of faces of $C$ and let $b:\mathcal{F}(C)\rightarrow 
\mathcal{B}(C)$ the function that assigns to each face its baricenter.

\begin{lemma}
\label{convergence} (a) If $\mathcal{F}(C)$ is compact then $b$ is a
homeomorphism.

(b) If $\mathcal{E}(C)$ is compact then a sequence of faces $F_{i}$
converges to a face $F$ if and only if $\mathcal{E}(F_{i})$ converges to $%
\mathcal{E}(F)$.
\end{lemma}

\begin{proof}
(a) The function that assigns to each compact convex set in $\mathbb{R}^{n}$
its baricenter is continuous, so $b:\mathcal{F}(C)\rightarrow \mathcal{B}(C)$
is a continuous bijective map from a compact Hausdorff space to a metric
space.

(b) The Hausdorff distance between two compact convex sets is bounded above
by the Hausdorff distance between their sets of extreme points.

If $F_{i}$ converges to $F$ but $\mathcal{E}(F_{i})$ doesn%
\'{}%
t converge to $\mathcal{E}(F)$ then there is a subsequence $\mathcal{E}%
(F_{i_{j}})$ that stays at distance at least $\varepsilon >0$ from $\mathcal{%
E}(F)$. For each $i_{j}$ there is an extreme point $p_{i_{j}}\in F_{i_{j}}$
whose distance from $\mathcal{E}(F)$ is larger than $\varepsilon $, or an
extreme point $q_{i}\in F$ whose distance from $\mathcal{E}(F_{i_{j}})$ is
larger than $\varepsilon $. If there is a convergent subsequence $%
p_{i_{k}}\rightarrow p\in F$ then $p$ is at distance at least $\varepsilon $
from $\mathcal{E}(F)$, so $p$ can't be an extreme point of $C.$

If there is a convergent subsequence $q_{i_{k}}\rightarrow q\in F$, take $%
p_{i_{k}}^{\prime }\in F_{i_{k}}$ with $p_{i_{k}}^{\prime }\rightarrow q$.
Eventually $\left\vert p_{i_{k}}^{\prime }-q_{i_{k}}\right\vert <\frac{%
\varepsilon }{2}$ so the distance from $p_{i_{k}}^{\prime }$ to $\mathcal{E}%
(F_{i_{k}})$ is at least $\frac{\varepsilon }{2}$, so $p_{i_{k}}^{\prime }$
is the center of a straight interval $I_{i_{k}}$ of length $\varepsilon $
contained in $F_{i_{k}}$. A convergent subsequence of these intervals yields
a straight interval centered at $q$ and contained in $F$, so $q$ can't be an
extreme point of $C$, contradicting the compacity of $\mathcal{E}(C)$.
\end{proof}

\begin{lemma}
\label{extension} If $C$ and $C^{\prime }$ are convex bodies with $\mathcal{F%
}(C)$ and $\mathcal{F}(C^{\prime })$ compact, then any continuous "face
preserving" map $\varphi :\mathcal{E}(C)\rightarrow \mathcal{E}(C^{\prime })$
extends naturally to a continuous map $\varphi :C\rightarrow C^{\prime }$.
\end{lemma}

\begin{proof}
$\varphi $ determines a function $\Psi :\mathcal{F}(C)\rightarrow \mathcal{F}%
(C^{\prime })$. $\Psi $ is continuous because by lemma \ref{convergence}, $%
F_{i}\rightarrow F$ implies $\mathcal{E}(F_{i})\rightarrow \mathcal{E}(F)$,
then uniform continuity of $\varphi $ on $\mathcal{E}(C)$ implies that $%
\varphi (\mathcal{E}(F_{i}))\rightarrow \varphi (\mathcal{E}(F))$ so by
definition $\mathcal{E}(\Psi (F_{i}))\rightarrow \mathcal{E}(\Psi (F))$ and
so $\Psi (F_{i})\rightarrow \Psi (F)$. So $\varphi $ can be extended to a
continuous function $\varphi :\mathcal{B}(C)\rightarrow \mathcal{B}%
(C^{\prime })$ as $b\circ \Psi \circ b^{-1}$ (recall that $\mathcal{E}%
(C)\subset \mathcal{B}(C)$). Now we can extend $\varphi $ to the interiors
of the faces of $C$ defining it linearly on rays, as follows.

For each point $a\in C,$ let $F(a)$ be the unique face of $C$ containing $a$
in its interior and let $b(a)$ be the baricenter of $F(a)$. Although $F(a)$
and $b(a)$ are not continuous functions of $a$ on all of $C$, they are
continuous on the union of the interiors of the faces corresponding to $h$%
-flats for each $h$. If $a\neq b(a)$ let $p(a)$ be the projection of $a$ to $%
\partial F(a)$ from $b(a)$ and let $\lambda (a)=\frac{\left\vert
a-b(a)\right\vert }{\left\vert p(a)-b(a)\right\vert }$ (or $0$ if $a=b(a)$)
so $a=(1-\lambda (a))b(a)+\lambda (a)p(a)$. Define $\varphi (a)=(1-\lambda
(a))\varphi (b(a))+\lambda (a)\varphi (p(a)).$

Assume inductively that $\varphi $ is continuous on the union of $\mathcal{B}%
(C)$ and the faces of $C$ of dimension less than $d$ (this set is closed
because the limit of faces of dimension less than $d$ has dimension less
than $d$). and let's show that for each sequence of points $a_{i}$ in the
interiors of faces of dimension $d$, $a_{i}\rightarrow a$ implies $\varphi
(a_{i})\rightarrow \varphi (a)$. We may assume that the $a_{i}$ are not
baricenters, so $p(a_{i})$ is well defined.

Case 1. $F(a_{i})\rightarrow F(a)$ then $b(a_{i})\rightarrow b(a)$ by the
continuity of $b$ on faces.

If $b(a)\neq a$ then $p(a_{i})\rightarrow p(a)$ and $\lambda
(a_{i})\rightarrow \lambda (a)$ so $\varphi (a_{i})=(1-\lambda
(a_{i}))\varphi (b(a_{i}))+\lambda (a_{i})\varphi (p(a_{i}))\rightarrow
(1-\lambda (a))\varphi (b(a))+\lambda (a)\varphi (p(a))=\varphi (a)$.

If $b(a)=a$ then $\lim b(a_{i})=\lim a_{i}$ but $p(a_{i})$ may not converge,
so consider a convergent subsequence $p(a_{i_{j}})$: If $\lim
p(a_{i_{j}})\neq $ $\lim b(a_{i_{j}})=\lim a_{i_{j}}$ then $\lim \lambda
(a_{i_{j}})=0$ so $\varphi (a_{i_{j}})=\varphi (b(a_{i_{j}}))+\lambda
(a_{i_{j}})\left[ \varphi (p(a_{i_{j}}))-\varphi (b(a_{i_{j}}))\right]
\rightarrow \varphi (b(a))+0=\varphi (a)$. If $\lim p(a_{i_{j}})=$ $\lim
b(a_{i_{j}})$ then $\lim \varphi (p(a_{i_{j}}))=\lim \varphi (b(a_{i_{j}}))$
(by continuity of $\varphi $ in the baricenters and faces of lower
dimension) and as $\varphi (a_{i_{j}})$ lies between them, $\lim \varphi
(a_{i_{j}})=\lim \varphi (b\circ F(a_{i_{j}}))=\varphi (b(a))=\varphi (a)$.

Case 2. $F(a_{i})\nrightarrow F(a)$, then for any convergent subsequence $%
F(a_{i_{j}})$ with limit a face $F\neq F(a)$, $a$ lies in $F$ and so $a$
must lie in $\partial F$, so $\left\vert a_{i_{j}}-p(a_{i_{j}})\right\vert
\rightarrow 0$ and $\lambda (a_{i_{j}})\rightarrow 1$, so $\lim \varphi
(a_{i_{j}})=\lim (1-\lambda (a_{i_{j}}))\varphi (b(a_{i_{j}}))+\lambda
(a_{i_{j}})\varphi (p(a_{i_{j}}))=\lim \varphi (p(a_{i_{j}}))=\varphi (a)$
(by continuity of $\varphi $ on the faces of lower dimension).
\end{proof}

\begin{theorem}
\label{main} Let $C$ be a convex body whose face lattice defines a
n-dimensional projective space.

If $n=2$, then $C$ has dimension $5,8,14$ or $26$.

If $n>2$ (or the space is desarguesian) there is a face-preserving
homeomorphism from $C$ to a section of a cone of Hermitian matrices over $%
\mathbb{R}$, $\mathbb{C}$, or $\mathbb{H}$.
\end{theorem}

\begin{proof}
First consider the case $n=2$. All the lines of a topological projective
plane $\mathcal{P}$ are homeomorphic because if $l$ is a line and $p$ is a
point not in $l$ then the projection $\phi :\mathcal{P}-p\rightarrow l$, $%
\phi (x)=(x\vee p)\wedge l$ is continuous and its restriction to each
projective line not containing $p$ is one to one. If the projective lines
are topological spheres then a famous result of Adams \cite[p.1278]{HIG},
shows that their dimension must be $d=0,1,2,4$ or $8$.

To compute the dimension of $C$ take 3 faces of rank 1, $F_{0}$, $F_{1}$ and 
$F_{2}$ so that $F_{1}$ and $F_{2}$ meet at a point $p$ not in $F_{0}$. The
projection $\phi :\mathcal{E}(C)-\{p\}\rightarrow \partial F_{0}$ extends to
a continuous map $\phi :\cup \left\{ F\text{ }|\text{ }F\text{ face of }C%
\text{, }p\notin F\right\} \rightarrow F_{0}$ whose restriction to each face
is one to one (see proof of lemma \ref{extension}). Now $U=\cup \left\{ IntF%
\text{ }/\text{ }F\text{ face of }C\text{, }p\notin F\right\} $ is an open
subset of $\partial C$ and the function $\Phi :U\rightarrow F_{0}\times
\left( \partial F_{1}-\{p\}\right) \times \left( \partial F_{2}-\{p\}\right) 
$ defined as $\Phi (x)=(\phi (x),\partial F(x)\wedge \partial F_{1},\partial
F(x)\wedge \partial F_{2})$ is continuous and bijective, so $U$ has the same
dimension as $F\times \partial F\times \partial F$, which is $3d+1$,
therefore $C$ has dimension $3d+2$. Note that the discrepancy between the
dimensions of the union of the boundaries of the faces ($2d$) and the union
of the faces ($3d+1$) arises because the boundaries of the faces overlap (as
the lines in a projective plane do) but the interiors of the faces are
disjoint. When $n>2$,\ there is a similar homeomorphism from an open subset
of $\partial C$ and a product $F_{0}\times \left( \partial
F_{1}-\{p\}\right) \times \left( \partial F_{2}-\{p\}\right) \times
...x\left( \partial F_{n}-\{p\}\right) $ where $F_{0}$ is a face of rank $%
r-1 $ and $F_{1},F_{2},...,F_{n}$ are faces of rank 1. So $\dim (C)=\dim
(F_{0})+rd+1$ and it follows by induction that $\dim C=\frac{n(n-1)}{2}d+n-1$%
.

Now assume that the projective space determined by$\ \mathcal{E}(C)$ is
desarguesian. Every topological desarguesian projective space is isomorphic
to a projective space over a \textit{topological} division ring $A$ (defined
on a line minus a point) and the isomorphism is a homeomorphism \cite[p.1261]%
{HIG}. By a classic result of Pontragin \cite[p.1263]{HIG} the only locally
compact, connected division rings are $\mathbb{R},$ $\mathbb{C}$ and $%
\mathbb{H}$ . So the projective space determined by $\mathcal{E}(C)$ is
isomorphic to $\mathbb{RP}^{r},\mathbb{CP}^{r}$ or $\mathbb{HP}^{r}$.
Therefore $\mathcal{E}(C)$ is isomorphic to $\mathcal{E}(C^{\prime })$ where 
$C^{\prime }$ is a section of a cone of Hermitian matrices, and so by \ref%
{extension} there is a face-preserving homeomorphism from $C$ to $C^{\prime
} $.
\end{proof}

\section{Face lattices defining affine spaces.}

Now let us consider closed (but not necessarily compact) convex sets in $%
\mathbb{R}^{n}$ whose faces\ meet as the subspaces of an affine space. An
abstract affine plane consists of a set of points and a set of lines so that
1) there are at least 2 points and 2 lines 2) every pair of points is
contained in a line and 3) given a line and a point not contained in it,
there is a unique line containing the point and parallel to the line. The
axioms of an abstract affine space are not so simple, but it is enough to
know that if $P$ is a projective space then the complement of a maximal flat
of $P$ is an affine space, and any affine space $A$ can be embedded in a
projective space in this fashion, by attaching to $A$ a point at infinity
for each parallelism class of affine lines.

Observe that if a closed convex set $C$ in $\mathbb{R}^{n}$ is non-compact,
it contains a ray (half of a euclidean line) and if $C$ contains a ray then
it contains all the parallel rays\ starting at points of $C$ (we say that $C$
contains an infinite direction). So if $C$ contains a line, $C$ is the
product of that line and a closed convex set $C^{\prime }$ of lower
dimension and the face lattice of $C$ \ and $C^{\prime }$ are isomorphic. So
from now on we will assume that $C$ doesn't contain lines.

It is easy to see that the faces of a polytope cannot determine an affine
space (the faces of rank $i$ would have dimension $i$, two parallel faces of
rank 1 generate a face of rank 2 with at least 4 vertices, but the sides of
a polygon don't define an affine plane).

\begin{example}
Let $C$ be a convex body in $\mathbb{R}^{n}$ whose faces determine a
projective space. Take a cone over $C$ and slice it with a hyperplane
parallel to a support hyperplane containing a maximal face. The result is a
closed non-compact convex set $C^{\prime }$ in $\mathbb{R}^{n}$ whose faces
determine an affine space. In particular, the cones of Hermitian matrices
have non-compact sections whose face lattice determines a real, complex or
quaternionic affine space or an octonionic affine plane.
\end{example}

$\mathbb{RP}^{n}$ can be seen as the space of lines through the origin in $%
\mathbb{R}^{n+1}$ or as the quotient of the unit sphere $\mathbb{S}^{n}$ (or
the sphere at infinity of $\mathbb{R}^{n+1}$) by the action of the antipodal
map. Identifying $\mathbb{R}^{n}$ with a hyperplane of $\mathbb{R}^{n+1}$
that doesn't contain the origin gives an embedding of $\mathbb{R}^{n}$ as a
dense open subset of $\mathbb{RP}^{n}$. The remaining points of $\mathbb{RP}%
^{n}$ correspond to lines through the origin in $\mathbb{R}^{n+1}$ that
don't meet the hyperplane, i.e., parallelism classes of lines in $\mathbb{R}%
^{n}$ (or pairs of antipodal points in the sphere at infinity)$.$ Define a
set in $\mathbb{RP}^{n}$ to be \textit{convex} if it is the image of a
convex set in $\mathbb{R}^{n}$ under one of these embeddings. As convex sets
in $\mathbb{RP}^{n}$ correspond to convex cones based at the origin of $%
\mathbb{R}^{n+1}$, convexity in $\mathbb{RP}^{n}$ doesn't depend on the
particular embedding, and a convex set in $\mathbb{RP}^{n}$ has the usual
properties of a convex set in $\mathbb{R}^{n}$.

Now if $C$ is a closed convex set in $\mathbb{R}^{n}$ that doesn't contain
lines, its closure $\overline{C}$ is a convex set in $\mathbb{RP}^{n}$. The
faces of $\overline{C}$ are the closures of faces of $C$ and their
intersections with the sphere at infinity modulo the antipodal map.

\begin{lemma}
In a closed convex set in $\mathbb{R}^{n}$ that has semi-modular face
lattice, two faces of rank 1 can share at most one direction, and it
corresponds to a ray.
\end{lemma}

\begin{proof}
We are considering convex bodies without lines. Suppose that two rank 1
faces $F_{1}$ and $F_{2}$ have a common direction, i.e., there are segments
of parallel euclidean lines $l_{1}$ and $l_{2}$ lying in $F_{1}$ and $F_{2}$%
. We may assume that $l_{i}$ goes through an interior point of $F_{i},$ so $%
l_{i}$ meets $\partial F_{i}$ in one or two extreme points. If $l_{1}$ or $%
l_{2}$ has two extreme points then there is a convex quadrilateral with
sides in $l_{1}$ and $l_{2}$ with 3 extreme points as vertices. The interior
of the quadrilateral lies in the interior of the rank 2 face generated by
the 3 extreme points, but the intersection of its diagonals lies in the rank
1 face generated by 2 extreme points, a contradiction. So $l_{i}$ meets $%
\partial F_{i}$ in only one point and so $F_{i}$ contains a ray $l_{i}^{+}$.
If $F_{1}$ and $F_{2}$ have two common directions, there is a common
direction which meets $F_{1}$ in 2 points, giving the same contradiction.
\end{proof}

\begin{lemma}
If the faces of a closed convex set $C$ in $\mathbb{R}^{n}$ define an affine
space, then each face representing a line contains a unique ray, and faces
representing parallel lines contain parallel rays.
\end{lemma}

\begin{proof}
We are assuming again that $C$ doesn't contain lines, so the points of the
affine space correspond to the extreme points of $C$. Each affine lines is
represented by\ the boundary of a convex set of dimension at least 2, which
is connected, so set of extreme points in $C$ is connected.

Let's first show that the faces representing affine lines cannot be compact.
Suppose that $C$ has a compact face $F$. Let $p$ and $q$ be two extreme
points in $F$ and let $q_{i}$ be a sequence of extreme points not in $F$
that converge to $q$. Let $F_{i}$ be the face generated by $p$ and $q_{i}$.
If $F_{i}$ is non-compact, it contains a ray $r_{i}$ through $p$. So $F_{i}$
contains the "parallelogram" determined by the interval $pq_{i}$ and the ray 
$r_{i}. $An infinite sequence of $l_{i}$'s would have a subsequence
converging to a ray $l$ through $q$, so the parallelogram determined by the
interval $pq$ and the ray $l$ would be contained in $\partial C$, so it
would have to be contained in a face of $C$, which would have to be $F$
because it contains $p$ and $q$. This contradicts the assumption that $F$ is
compact and shows that if $q_{i}$ is sufficiently close to $q$, the face
generated by $p$ and $q_{i}$ is compact. Now take a face $F^{\prime }$ that
doesn't meet $F$ (i.e., $F$ and $F^{\prime }$ represent parallel affine
lines). As $F$ is compact and $F^{\prime }$ is closed in $\mathbb{R}^{n}$,
there is an $\varepsilon $ neighborhood of $F$ that doesn't intersect $%
F^{\prime }$. By the previous argument there is a point $q_{i}$ not in $F$
so that the face $F_{i}$ generated by $p$ and $q_{i}$ is contained in the $%
\varepsilon $ neighborhood of $F$. So $F_{i}$ doesn't meet $F^{\prime }$,
but $F$ was supposed to be the only face containing $p$ and disjoint from $%
F^{\prime }$.

This proves that $F$ is non-compact, so it contains rays. Let's show that\
two faces representing parallel affine lines contain parallel rays. Let $p$
be an extreme point outside $F$, so $F$ and $p$ generate a face $H$
representing an affine plane. There are extreme points $%
p_{0},p_{1},p_{2},... $ in $F$ so that the sequence of intervals $p_{0}p_{i}$
converges to a ray $l_{+}$ contained in $F$ (because $F$ is closed). The
sequence of intervals $pp_{i}$ lie in $\partial H$ and converge to a ray $%
m_{+}$ parallel to $l_{+}$ and containing $p$ so (as $H$ is closed) $m_{+}$
is contained in a face $G$ of $\partial H$ representing an affine line. As
two faces that contain parallel rays cannot meet at a single point, $G$
doesn't meet $F$ so (as $F$ and $G$ are contained in $H$) $G$ represents the
affine line parallel to $F$ through $p$. If $F$ has nonparallel rays, one
can construct as before two nonparallel rays $l_{+}$ and $l_{+}^{\prime }$
in $F$ and faces $G$ and $G^{\prime }$ through $p$ and containing rays $m_{+}
$ and $m_{+}^{\prime }$. The uniqueness of parallel affine lines implies
that $G=G^{\prime }$, so $F$ and $G$ have more than one common direction,
contradicting the previous lemma.
\end{proof}

This shows that if the faces of a closed convex set $C$ define an affine
space, $C$ is non-compact. One can show that if the faces of a closed convex
set $C$ (containing no lines) define a projective space, $C$ must be
compact. For this, one has to give a topology to the space of closed convex
sets in $\mathbb{R}^{n}$ that makes it locally compact, show that this makes 
$\mathcal{E}(C)$ into a locally compact projective space, and observe that
these spaces are necessarily compact.

\begin{theorem}
If $C$ is a closed convex set in $\mathbb{R}^{n}$ whose faces determine an
affine space, there is a projective transformation in $\mathbb{RP}^{n}$
taking $\overline{C}$ to a compact convex set in $\mathbb{R}^{n}$ whose
faces determine a projective space. If the space is desarguesian, there is a
face-preserving homeomorphism from $C$ to a non-compact slice of a cone of
Hermitian matrices.
\end{theorem}

\begin{proof}
We need to show that if the face lattice of $C$ determines an affine space,
the face lattice of $\overline{C}\subset \mathbb{RP}^{n}$ determines its
projective completion. The faces of $C$ representing affine lines are
non-compact, and two of them share an infinite direction in $\mathbb{R}^{n}$
if and only if they represent parallel affine lines. The closure $\overline{C%
}\subset \mathbb{RP}^{n}$ contains one point at infinity for each infinite
direction in $C$, so $\overline{C}$ contains an extreme point at infinity
for each class of faces of $C$ representing parallel lines. This corresponds
precisely with the definition of the projective completion of the affine
space. Now the result for $C$ follows by applying theorem \ref{main} to $%
\overline{C}$.
\end{proof}

\section{Projective planes and the case $d=1$.}

The face lattice of a convex body $C$ (not a triangle) determines a
projective plane if every pair of extreme points is contained in a proper
face and every pair of faces with more than one point meet. By theorem \ref%
{main} this projective plane is compact and connected, so for some $d\in
\left\{ 1,2,4,8\right\} $, all the faces of $C$ have dimension $d+1$ and $C$
has dimension $n=3d+2$.

\begin{lemma}
\label{spans} A $d+1$ dimensional subspace of $\mathbb{R}^{n}$ is the span
of a face of $C$ if and only if it meets all the spans of faces of $C$ .
\end{lemma}

\begin{proof}
Let $S$ be an affine subspace that intersects $span(F)$ for every $F\in 
\mathcal{F}_{1}$, the set of faces of rank 1. Then $\left\{ F\in \mathcal{F}%
_{1}\text{ }|\text{ }\dim (S\cap span(F))\geq i\right\} $ is closed in $%
\mathcal{F}_{1}$ for each $i$.

Case 1. $d=1.$ We claim that if $S$ is not the span of a face then $S$
cannot intersect $span(F)$ in more than one point. For, if $S\cap span(F)$
contains a line, then $span(S\cup F)$ is 3 dimensional. Take an extreme
point $p\notin span(S\cup F)$ and let $F_{1}$ and $F_{2}$ be 2 faces
containing $p$, and meeting $F$ at points $p_{1}$ and $p_{2}$ not in $S$. If 
$p_{1}^{\prime }$ and $p_{2}^{\prime }$ are points in $S\cap span(F_{1})\ $%
and $S\cap span(F_{2})\ $respectively, then $p$, $p_{i}$ and $p_{i}^{\prime
} $ are not aligned (otherwise $p$ would be in the span of $S\cup span(F)$)
and so the span of $p$, $p_{i}$ and $p_{i}^{\prime }$, which is $span(F_{i})$%
, is contained in $span(p$ $\cup S\cup F)$. But $span(p$ $\cup S\cup F)$ is
4 dimensional, so it cannot contain the 3 faces $F$, $F_{1}$ and $F_{2}$
because if it did, it would contain each face that meets $F$, $F_{1}$ and $%
F_{2}$ at $3$ different points, but every face is a limit of such faces, so
it would contain all the faces of $C$, but $C$ has dimension 5. This shows
that $S$ intersects each $span(F)$ at exactly one point, and so $S$ contains
at most one extreme point of $C$. The function $I:\mathcal{F}_{1}\rightarrow
S$ that maps each face $F_{i}$ to the point of intersection of $span(F)$
with $S$ is continuous, and as the spans of faces meet only at extreme
points, $I$ only fails to be injective on the faces containing the extreme
point in $S$ (if any). But $\mathcal{F}_{1}$ is a 2-dimensional closed
surface in which the faces that contain an extreme point form a closed
curve, and there are no continuous maps from a closed surface to the plane
that fail to be injective only along a curve.

Case 2. $S$ doesn't contain extreme points of some face $F$. Choose $F$ that
minimizes the dimension of the subspace $S\cap span(F)$. Then for every $%
F^{\prime }$ in a neighborhood of $F$, $S\cap span(F)$ is a subspace of
minimal dimension and with no extreme points. If $S^{\prime }$ is the
orthogonal complement of $S\cap span(F)$ in $S$ then $S^{\prime }$
intersects $span(F)$ in one point for all $F^{\prime }$ in a smaller
neighborhood $V$ of $F.$ Then the function $I:V\rightarrow S^{\prime }$ that
maps $F^{\prime }$ to $S^{\prime }\cap span(F)$ is continuous, and it is
injective as the spans of faces only meet at extreme points. But an
injective map between manifolds can only exist when the domain has dimension
no larger than the target so $2d\leq \dim S^{\prime }\leq \dim S\leq d+1$,
so $d=1$ and we are in case 1.

Case 3. $S$ contains extreme points of each face $F.$ As $C$ is convex,
either $S\cap C\subset \partial C$ or $\partial _{S}(S\cap C)=S\cap \partial
C$. In the first case $S\cap C$ is contained in a face $F_{1}$ of $C$ and so
either $S\cap C=F_{1}$ (so $F_{1}\subset S$) or there is an extreme point $p$
of $F_{1}$ not contained in $S$, but then a face $F_{2}$ that meets $F_{1}$
at $p$ doesn't meet $F_{1}\cap S\supset S\cap C$ so $S$ doesn't contain
extreme points of $F_{2}$.

Let $p$ be an extreme point not in $S$, and consider the set $\mathcal{F}%
_{1}^{p}$ of faces of rank 1 containing $p$. If $F$\ and $F^{\prime }$ are
distinct faces in $\mathcal{F}_{1}^{p}$, $S\cap F$ and $S\cap F^{\prime }$
are disjoint. Choose $F$ so that $S\cap F$ has minimal dimension, then for
all $F^{\prime }$ in some neighborhood $V$ of $F$, $S\cap F^{\prime }$ has
the same dimension and the map $I_{B}:\mathcal{F}_{1}^{p}\cap V$ $%
\rightarrow S\cap \partial C$ that sends $F^{\prime }$ to the baricenter of $%
S\cap F^{\prime }$ is continuous and injective. As $I_{B}$ is a map between
manifolds, $d=\dim \mathcal{F}_{1}^{p}\leq \dim S\cap \partial C\leq d$ and
so by domain invariance the image of $I_{B}$ is an open subset of $S\cap
\partial C=\partial _{S}(S\cap C)$. This implies that for each $F^{\prime
}\in \mathcal{F}_{1}^{p}\cap V$, $S\cap F^{\prime }$ consists of one point
(if a face of a convex set has more than 1 point, its baricenter is
arbitrarily close to points in the boundary that are not baricenters of
other faces, namely, the points in the face) and so, by hypothesis, $S\cap
F^{\prime }$ is an extreme point of $C$.

So part of the boundary of $S\cap C$ in $S$ is strictly convex, therefore
the line segment joining two extreme points in it lies in $Int_{S}(S\cap C)$%
, but that line segment lies in the face of $C$ containing the 2 extreme
points, so it must lie in $S\cap \partial C=\partial _{S}(S\cap C)$, a
contradiction.
\end{proof}

\begin{lemma}
The boundaries of the faces of rank 1 of $C$ are semi-algebraic sets. If $%
d=1 $, they are conic sections.
\end{lemma}

\begin{proof}
By lemma \ref{spans}, the set $\mathcal{S}$ of spans of faces of $C$ is the
same as the set of $d+1$-dimensional subspaces of $\mathbb{R}^{n}$ that
intersect every element of $\mathcal{S}$. The set of all $d+1$-dimensional
affine subspaces of $\mathbb{R}^{n}$ forms a real algebraic variety and the
condition that the subspaces meet a fixed subspace is algebraic, so (by the
finite descending chain condition) there is a finite family of spans $%
S_{1},S_{2},...,S_{m}\in \mathcal{S}$ such that $S\in \mathcal{S}$ if it
intersects these $S_{i}$'s (see \cite{AG}).

Now for $(x_{1},x_{2},...,x_{m})\in S_{1}\times S_{2}...\times S_{m}$, the
subspace $span(x_{1},...,x_{m})$ has dimension at least $d+1$ (otherwise it
would be contained in two subspaces of dimension $d+1$ that meet each $S_{i}$%
, so they would both be in $\mathcal{S}$, but two spans can only meet in 1
point). So $span(x_{1},...,x_{m})$ lies in $\mathcal{S}$ if and only if its
dimension is $d+1$, and this happens if and only if some determinants (given
by polynomials on $x_{1},...,x_{m}$ ) vanish. Therefore the set $X=\left\{
(x_{1},x_{2},...,x_{m})\in S_{1}\times S_{2}...\times S_{m}\text{ }|\text{ }%
span(x_{1},...,x_{m})\in \mathcal{S}\right\} $ is real algebraic, as is the
set $X^{p}$ formed by the elements of $X$ that contain a fixed point $p$. If 
$F_{1}$ is the face in $S_{1}$ and $p$ is an extreme point of $C$ outside $%
F_{1}$ then $\partial F_{1}$ consists of the intersections of $S_{1}$ with
the elements of $\mathcal{S}$ containing $p$. So $\partial F_{1}$ is the one
to one projection of the algebraic set $X^{p}$ to $S_{1}$, so $\partial
F_{1} $ is at least semi-algebraic.

\medskip\centerline{\includegraphics{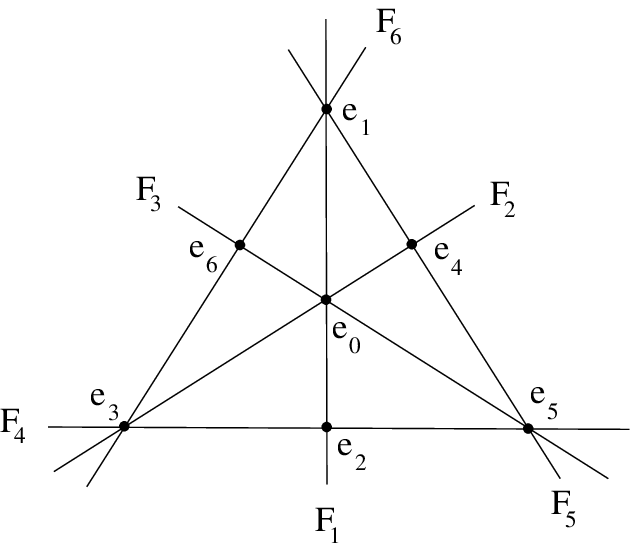}} \centerline{\sr Figure 1}
\medskip

Now assume $d=1$ so $n=5$. Every projective plane has 7 points and 6 lines
so that each line contains 3 points as in figure 1, so $C$ has 7 extreme
points and 6 faces intersecting in that way. The 7 points are in general
position in $\mathbb{R}^{5}$ because as each face of $C$ is spanned by 3
points, the span of any 6 of those points contains the span of 3 faces,
which is all of $\mathbb{R}^{5}$. Therefore we may assume (by applying a
projective transformation) that the 7 points are $%
p_{0}=(0,0,0,0,0),p_{1}=(1,0,0,0,0),...,p_{5}=(0,0,0,0,1),p_{6}=(1,1,1,1,1)$%
. Let $S_{i}$ be the plane spanned by the face $F_{i}$. A plane $S$ that
intersects $S_{1},S_{2}$ and $S_{3}$ has a parametrization $%
(x,y,z,v,w)=r(a,b,0,0,0)+s(0,0,c,d,0)+t(e,e,e,e,f)$ with $r+s+t=1$. $S$
intersects $S_{4},S_{5}$ and $S_{6}$ only if three systems of linear
equations in $r,s,t$ represented by the following matrices have nontrivial
solutions:

\bigskip 

\noindent $\left\vert 
\begin{array}{ccc}
a & 0 & e \\ 
0 & d & e \\ 
b-1 & c-1 & 2e+f-1%
\end{array}%
\right\vert \left\vert 
\begin{array}{ccc}
b & 0 & e \\ 
0 & c & e \\ 
a-1 & d-1 & 2e+f-1%
\end{array}%
\right\vert \left\vert 
\begin{array}{ccc}
b & 0 & e-f \\ 
0 & d & e-f \\ 
a-1 & c-1 & 2e-f-1%
\end{array}%
\right\vert $

\bigskip 

\noindent As the determinants of these matrices are linear functions on the
variables $e$ and $f$, they vanish simultaneously if and only if the matrix
of this new system has determinant 0:

\bigskip

$\det \left\vert 
\begin{array}{ccc}
-ac+2ad-bd+a+d & ad & -ad \\ 
-ac+2bc-bd+b+c & bc & -bc \\ 
-ad-bc+2bd+b+d & ad+bc-bd-b-d & -bd%
\end{array}%
\right\vert =0$

\bigskip

\noindent This determinant factors as the product of a linear and a
quadratic function of $a$ and $b$ (with coefficients in $c$ and $d$). Since
the boundary of the face $F_{1}$ is formed by the intersections of $S_{1}$
with the planes that meet all $S_{i}$'s and go through a fixed point in the
boundary of $F_{2}$ (this corresponds to fixing $c$ and $d$), the boundary
of $F_{1}$ is contained in the union of a line and a conic. As the boundary
of $F_{1}$ is strictly convex, it must be the conic.
\end{proof}

\bigskip

\begin{theorem}
\label{dim5} All convex bodies in $\mathbb{R}^{5}$ with modular and
irreducible face lattice are projectively equivalent.
\end{theorem}

\begin{proof}
Let $C$ and $C^{\prime }$ be two such bodies. Take extreme points $%
p_{0},p_{1},...,p_{6}$ and faces $F_{1},...F_{6}$ of $C$ as in figure 1.
Pick an extreme point $p_{0}^{\prime }$ in $C^{\prime }$ and two faces $%
F_{1}^{\prime }$ and $F_{2}^{\prime }$ of $C^{\prime }$ intersecting at $%
p_{0}^{\prime }$. Let $S_{i}$ be the span of $F_{i}$. As the faces of $C$
and $C^{\prime }$ are conics, there are linear transformations from $S_{1}$
to $S_{1}^{\prime }$ taking $F_{1}$ to $F_{1}^{\prime }$ and from $P_{2}$ to 
$P_{2}^{\prime }$ taking $F_{2}$ to $F_{2}^{\prime }$. Together, they define
a linear transformation $l$ from $span(F_{1}\cup F_{2})$ to $%
span(F_{1}^{\prime }\cup F_{2}^{\prime })$. Let $p_{i}^{\prime }=l(p_{i})$
for $i=1,...,4$. The faces $F_{4},F_{5},F_{6}$ are generated by unique pairs
of $p_{i}$'s with $i\leq 4$. Let $F_{4}^{\prime },F_{5}^{\prime
},F_{6}^{\prime }$ be the faces generated by the corresponding pairs of $%
p_{i}^{\prime }$s. Finally, let $p_{5}^{\prime }=S_{4}^{\prime }\cap
S_{5}^{\prime }$, let $F_{3}^{\prime }$ be the face generated by $%
p_{0}^{\prime }$ and $p_{5}^{\prime }$ and let $p_{6}^{\prime
}=S_{3}^{\prime }\cap S_{6}^{\prime }$. The linear transformation $l$ can be
extended to a projective transformation $\rho $ in $\mathbb{RP}^{5}$ that
takes $p_{5}$ to $p_{5}^{\prime }$ and $p_{6}$ to $p_{6}^{\prime }$. As $%
\rho $ sends each $p_{i}$ to $p_{i}^{\prime }$, it sends each $S_{i}$ to $%
S_{i}^{\prime }$, so it sends each plane in $\mathbb{R}^{5}$ intersecting
every $S_{i}$ to a plane intersecting every $S_{i}^{\prime }$. Since by
construction $\rho $ takes those planes that meet $\partial F_{1}$ and $%
\partial F_{2}$ to planes that meet $\partial F_{1}^{\prime }$ and $\partial
F_{2}^{\prime }$, lemma \ref{spans} implies that $\rho $ maps spans of faces
of $C$ to spans of faces of $C^{\prime }$ and therefore it maps faces to
faces.
\end{proof}

\bigskip

Question 1: \textit{Are all the convex bodies whose face lattices determine
classical projective spaces projectively equivalent to sections of cones of
hermitian matrices?}

\bigskip

Question 2: \textit{Can two convex bodies of the same dimension define non
isomorphic projective planes (so they are not related by a face-preserving
homeomorphism)?}\bigskip

In dimensions\textit{\ }$8$ and $14$ this is equivalent to ask if the
projective planes are always desarguesian. In dimension $26$ there might be
enough space for non-equivalent non-desarguesian examples.

\end{document}